\documentclass[]{article}

\title{Isometric embedding and spectral constraints for weighted graph metrics}
\author{J. Cheng\footnote{Department of Mathematics, UT Austin, jeffrey.cheng@utexas.edu},  I.M.J. McInnis\footnote{Department of Mathematics, Princeton University, imj@princeton.edu},  M. Yee\footnote{Department of Mathematics, Brown University, matthew\_w\_yee@brown.edu}}
\date{\today}
\usepackage{amsthm,amssymb,amsmath,tikz, enumitem}
\usetikzlibrary{decorations.pathreplacing,calligraphy}
\newtheorem{thm}{Theorem}[section]
\newtheorem{coro}[thm]{Corollary}
\newtheorem{lmm}[thm]{Lemma}
\newtheorem{defn}[thm]{Definition}
\newtheorem{prop}[thm]{Proposition}
\newtheorem{rem}[thm]{Remark}
\newtheorem{conj}[thm]{Conjecture}

\usepackage[citestyle=numeric,bibstyle=numeric,backend=bibtex,giveninits=true,doi=false,eprint=false,isbn=false,maxnames=6]{biblatex}
\begin{document}
\maketitle


\begin{abstract}
    A weighted graph $\phi G$  
    encodes a finite metric space $D_{\phi G}$. When is $D$ totally decomposable? When does it embed in $\ell_1$ space? When does its representing matrix have $\leq 1$ positive eigenvalue? 
We give useful lemmata and prove that these questions can be answered without examining $\phi$ if and only if $G$ has no $K_{2,3}$ minor. 
We also prove results toward the following conjecture. $D_{\phi G}$ has $\leq n$ positive eigenvalues for all $\phi$, if and only if $G$ has no $K_{2,3,...,3}$ minor, with $n$ threes.\footnote{This work was completed during the 2021 summer REU in matrix analysis at the College of William \& Mary, supervised by Prof. Charles R. Johnson. This work was supported by the National Science Foundation Grant DMS \#0751964. We thank Professor Johnson; the second author thanks Alan Chang.}  
\end{abstract}

\section{Introduction}
This paper considers simple, undirected, connected, weighted graphs presented as $\mathcal{G}\equiv \phi G \equiv (V(G),\, E(G),\, \phi: E(G)\rightarrow [0,\infty))$. For convenience, label the elements of $V(G)$ by $v_1,...,v_n$. For a path $p$ in $\mathcal{G}$, the quantity $\sum_{e\in E(p)} \phi(e)$ is called the \emph{length} and is denoted $|p|_{\phi G}$.
Letting $P_{ij}$ denote the set of paths in $G$ with endpoints $v_i,v_j$, the \emph{distance} $d_{\phi G}(v_i,v_j)$ 
between vertices $v_i,v_j$ is $\min_{p\in P_{ij}} |p|_{\phi G}$. This yields a finite metric space $(V(G), d_{\phi G})$, which we represent as a the \emph{distance matrix} $D\equiv D_{\phi G}$ whose $i,j$ entry $d_{ij}$ equals $d_{\phi G}(v_i,v_j)$. (To consider $\mathcal{G}$ ``unweighted''––i.e., with $\phi(e)=1$ for all $e$––one may instead write $D_G$.) None of our investigated properties vary under conjugation by a permutation matrix. (Otherwise we'd be studying graphs equipped with edge weights and vertex labels too.) Thus, we can freely move between ``matrix'' and ``finite (semi)metric space.'' 
The broad question is ``how does $G$ 
constrain the properties of $D$?"

From the matrix perspective, we consider $D$'s spectrum, i.e. its list of eigenvalues, which we write as a multiset $\sigma(D)$. A few spectral properties are obvious, but like most ``inverse eigenvalue problems,'' this soon proves intractable. (It is easier for unweighted graphs. See, for instance, \cite{LIN20131662, AOUCHICHE2014301,ATIK2015256,EDELBERG197623,GRAHAM197860,ZHANG201730}. The exciting and application-rich \cite{doi:10.1080/03081087.2020.1803187} contextualizes $D_G$ and $\sigma(D_G)$ in spectral graph theory broadly by defining a ``generalized distance matrix.'')

We study the spectrum's coarser cousin, inertia. The \emph{inertia} $i(A)$ of a matrix $A$ is a triplet of natural numbers $(i_+(A), i_0(A), i_-(A))$, denoting (respectively) the number of positive, zero, and negative eigenvalues of $A$. In the language of \cite{ChengJohnsonMcInnisYee}, $i(D_{\phi G})$ is the ``distance inertia'' of the weighted graph $\phi G$. Along with \cite{MR2133282}, that is the only other paper we know of that studies distance inertia for weighted graphs. Most study has focused on the distance inertia of unweighted graphs.\footnote{Even more extensively studied is the ``distance energy'' or ``$D$-energy.'' Instead of abstracting to the signs of $D$'s eigenvalues, one abstracts to their magnitude. The $D$-energy $DE(G)$ is defined as $\sum_{\lambda \in \sigma(D_G)}|\lambda|$. It's uninteresting for weighted graphs unless we somehow constrain $\phi$. One natural constraint, of which the ``unweighted'' study is a special case, might be to consider the case where $\sum_{e\in E(G)} \phi(e)=|E(G)|$. We are too far afield to give an appropriate survey, but the study of $DE(G)$ began in \cite{indulal2008distance} and has since intertwined with the study of $i(D_G)$ and $\sigma(D_G)$.} See \cite{MR289210, MR2133282, MR3056981,ZHANG2014108, TIAN2017470}. Our techniques have nothing in common with prior studies. 

Our specific question is, ``Given $n$, for what graphs $G$ is $i_+(D_{\phi G})$ always $\leq n$, regardless of $\phi$?'' We answer this question for $n=1$ but only conjecture about higher $n$. Unweighted distance inertia has been studied and characterized for cactus graphs \cite{MR3056981}, and weighted distance inertia has been studied for trees \cite{MR2133282}. 
This is all a specialization of the inverse eigenvalue problem for hollow, symmetric, nonnegative matrices. See  \cite{CFJK}, which observes that positive eigenvalues are rare in such matrices. By computer experimentation, they seem to be even rarer in distance matrices. See  Conjecture \ref{WeakConjecture}.

From the ``semimetric'' perspective, we consider the metric spaces in which $Dd$ can be embedded. (Whenever we speak of ``embedding,'' we mean a distance-preserving inclusion.) Though fascinating, this is too broad. We restrict ourselves to the single metric space $\ell_1$. (The points in $\ell_1$ are the real series that are nonzero at finitely many places; the metric comes from the norm $||(a_n)||=\sum |a_n|$. In other words, $\ell_1$ is $\mathbb{R}^{k}$ for arbitrarily high $k$ with the taxicab/Manhattan/$\ell_1$ metric––$D$ embeds in a finite-dimensional Manhattan space if and only if it embeds in $\ell_1$.) So our question is ``For what graphs $G$ is $D_{\phi G}$ always embeddable in $\ell_1$, regardless of $\phi$?'' 

The space $\ell_1$ has proven itself specially interesting for geometry, combinatorics, and theoretical computer science (\cite{MR2841334} is a good survey). The $\ell_1$ embedding of finite metrics (and graph metrics in particular) has also been extensively studied, such as in \cite{Arora, Chekuri, CutsTrees, Newman2010FiniteVS}.  The combinatorially interesting structure of $\ell_1$ gives rise to another property of a metric: being ``totally decomposable.'' This technical condition has been used in hundreds of papers in discrete mathematics and mathematical biology too various to survey here. It was first defined by Bandelt and Dress \cite{MR1153934}, who also originated many of the facts we'll use to prove our main result (though we usually cite the more accessible Deza and Laurent \cite{MR2841334}). That main result is this.
\newtheorem*{maintheorem}{Theorem \ref{k23theorem}}
\begin{maintheorem}
Let $G$ be a graph. The following are equivalent:
	\begin{enumerate}[label=(\roman*)]
		\item $G$ has no $K_{2,3}$ (graph) minor.
		\item $\nexists \phi: E(G) \rightarrow [0,\infty)$, $c\in \mathbb{R}$  s.t. $D_{cK_{2,3}}$ is a (distance) minor of $D_{\phi G}$.
		\item $\forall  \phi: E(G) \rightarrow [0,\infty)$, $D_{\phi G}$ is totally decomposable.
		\item $\forall \phi: E(G) \rightarrow [0,\infty)$, $D_{\phi G}$ is $\ell_1$-embeddable.
		\item $\forall \phi: E(G) \rightarrow [0,\infty)$, $i_+(D_{\phi G})=1$. 
	\end{enumerate}
\end{maintheorem}
(In the language of \cite{ChengJohnsonMcInnisYee}, this constitutes a characterization of all universally weakly unipositive graphs.)

Unfortunately, the tools from Bandelt and Dress don't seem to help us in generalizing to the problem's obvious variants. For instance, we can  substitute any $\ell_p$ for $\ell_1$ or ``$\leq n$ positive eigenvalues'' for ``$\leq 1$ positive eigenvalues.'' (Swapping ``positive'' for ``nonnegative'' or ``negative'' is dull, in light of Corollary \ref{WhatIsInteresting} and Remark \ref{OnlyPositivesAreInteresting}.)

As implied by Theorem \ref{metricembeddingisminorclosed}, the set of graphs that are always $\ell_p$-embeddable can always be characterized by a finite list of excluded minors. For $p=1$, the list is $\{K_{2,3}\}$ by Theorem \ref{k23theorem}. For $p=\infty$, the list is empty (recalling that any finite metric space embeds isometrically in $\ell_\infty$). However, since $p$ is a continuous parameter, and the set of excluded minors is a discrete, qualitative parameter, it seems quite hard to fill in all the intermediate values of $p$. 

The set of graphs whose $D_{\phi G}$ have $\leq n$ positive eigenvalues for all $\phi$ is probably more tractable, since $n$ is a discrete parameter. Our Corollary \ref{InertialRestrictionMinorClosed} says that they too can be characterized by forbidden minors, and Lemma \ref{ktwothreesforbidden} says that the graph $K_{2,3,...,3}$ (where there are $n$ threes) is among them. In Conjecture \ref{StrongConjecture}, we guess that it's alone. 

We begin by proving some useful facts about the nature of $D_{\phi G}$ and the relationship of its spectrum and metric embeddability to the minoring operations on $G$. Then we build and deploy the specific machinery for Theorem \ref{k23theorem}. We conclude with our conjectures. For brevity, we omit many standard definitions. 

\section{General Results}
\begin{prop}\label{ObviousFacts}
    As long as $\phi$ isn't uniformly $0$, the following hold of $D\equiv D_{\phi G}$. The eigenvalues of $D$ are real and sum to 0. $D$ is irreducible. The spectral radius $\rho(D)$ of $D$ is a simple eigenvalue of $D$ whose associated eigenvector has nonnegative entries. 
    For $d_{ij}$ the entries of $D$, we have 
            \[\min_i \sum_j d_{ij}\leq \rho(D) \leq \max_i \sum_j d_{ij}. \]
    \end{prop} 
\begin{proof}
    $D$ is symmetric with diagonal uniformly $0$, so its eigenvalues are real and sum to $0$. $D$ is irreducible because a symmetric matrix is irreducible unless it's a block-sum of matrices, which (by the triangle inequality) $D$ cannot be, unless $\phi=0$ everywhere, which we have forbidden. The rest follows from the Perron-Frobenius Theorem, since $D$ is nonnegative and irreducible. 
\end{proof}

\begin{lmm}\label{edgecontracting}
Consider $\phi G$, with $\phi(e)=0$ for some $e\in E(G)$. Define $\phi'G'$ from $G$ by contracting $e$ and define $\phi':=\phi_{E(G')}$. Then $\sigma(D_{\phi G})$ is just $\sigma(D_{\phi'G'})$ with an additional 0.
\end{lmm}
\begin{proof}
    Assume without loss of generality that  $e=\overline{v_{|V(G)|-1}v_{|V(G)|}}$. Then we see that, by adding to $D_{\phi' G'}$ a duplicate of its last column, a duplicate of its last row, and a $0$ in the bottom-right corner, we obtain $D_{\phi G}$. Then each eigenvector of $D_{\phi' G'}$, with a 0 appended to the end, is an eigenvector of $D_{\phi G}$ with the corresponding eigenvalue. Furthermore, this construction hasn't increased the rank, so the new eigenvalue must be 0.
\end{proof}

\begin{lmm}\label{edgedeleting}
Consider some $\phi G$. Define $G'$ from $G$ by adding an edge $e'$. Define $\phi'$ as equal to $\phi$ everywhere, except on $e'$, where it is $\geq (|V(G)|-1)\max_{e\in E(G)} \phi(e)$. Then $D_{\phi' G'} =D_{\phi G}$. (So their respective spectra are equal.)
\end{lmm}
\begin{proof} Let $p_{uv}$ be the path of least length between $u$ and $v$ in $G$. Since $G$ is connected and $p_{uv}$ contains at most $|V(G)|-1$ edges, $|p_{uv}|_{\phi G} \leq (|V(G)|-1)\max_{e\in E(G)} \phi(e)$. Any path from $u$ to $v$ in $G'$ either uses $e'$ or is in $G$; if it uses $e'$, it is by definition at least as long as $P_{uv}$, so it can be ignored. Thus, $d_{\phi G}(u,v)=d_{\phi' G'}(u,v)$, so $D_{\phi'G'}=D_{\phi G}$. 
\end{proof}

\begin{thm}\label{InertiaMinors}
    Let $G$ be a minor of $H$. Consider any $\phi: E(G)\rightarrow [0,\infty)$ and any real number $r$. Say $\sigma(D_{\phi G})=\{\lambda_1,\lambda_2,...,\lambda_n\}$. Then there exists a $\psi: E(H)\rightarrow [0,\infty)$ such that $\sigma(D_{\psi H})=\{r\lambda_1,r\lambda_2,...,r\lambda_n,0,....,0\}$.
\end{thm}
\begin{proof}
    By definition, $G$ can be reached from $H$ by contracting and deleting edges (since all our graphs are connected). For the edges $e\in E(H)$ of $H$ that will be contracted, set $\phi'(e)=0$. For the edges that will be deleted, set $\phi'(e)=N$, where $N$ is sufficiently large. For all other edges, set $\phi'(e)=r\phi(e)$. Then apply Lemma \ref{edgecontracting}, Lemma \ref{edgedeleting}, and the fact that when a matrix's entries are multiplied by a scalar, so are its eigenvalues.
\end{proof}

\begin{coro}\label{InertialRestrictionMinorClosed}
    The family ``graphs such that $D_{\phi G}$ always has $\leq n$ positive eigenvalues'' is minor-closed. It is therefore characterized by a finite set of forbidden minors.
\end{coro}
\begin{proof}
    The first sentence follows from Theorem \ref{InertiaMinors}. The second sentence uses the celebrated Robertson-Seymour Theorem \cite{ROBERTSON2004325}. 
\end{proof}

\begin{coro}\label{WhatIsInteresting}
    If some $D_{\phi G}$ (for $|V(G)|=n$) has inertia $(a,b,c)$, then for any $m\in \mathbb{Z}^+$ there exists a $\psi H$ with $|V(H)|-|V(G)|=|E(H)|-|E(G)|=m$ and $D_{\psi H}$ having inertia $(a,b+m,c)$. Also,
    for any $G$, there exists a $\phi$ such that $D_{\phi G}$ has inertia $(1,0,|V(G)|-1)$. Also, every $G$ with $n$ vertices can have $i(D_{\phi G})$ attain the values $(0,n,0)$ and $(1,k,n-k-1)$ for $0\leq k\leq n-2.$
\end{coro}
\begin{proof}
    For the first part, it suffices to add leaves with edge weight 0, invoking Lemma \ref{edgecontracting}. For the second part, every graph has a spanning tree minor, and we invoke the theorem of \cite{MR289210} that every unweighted tree has distance inertia $(1,0,n-1)$. For the third part, consider the all-0 weighting and the $n-k$-vertex minors of this spanning tree. 
\end{proof}
\begin{rem}\label{OnlyPositivesAreInteresting}
    This corollary sharpens our inertial questions. It indicates that restricting the number of zero eigenvalues isn't especially interesting, nor is restricting the number of negative eigenvalues. (As \cite{CFJK} commented, negative eigenvalues are the overwhelming default for hollow, symmetric, nonnegative matrices.) By extension, neither is it interesting to consider ``nonnegative,'' ``nonpositive,'' or ``nonzero.'' Only the restriction of positive eigenvalues will be rich. Furthermore, we've identified a set of inertial options available to all $n$-vertex graphs. The most natural question is, when are these the \textit{only} options? That can be rephrased precisely as, ``when is $i_+(D_{\phi G})$ guaranteed to be $\leq 1$?'' This is answered by Theorem \ref{k23theorem}.
\end{rem}

\begin{thm}\label{metricembeddingisminorclosed}
Let $Q$ be a metric space (or family of metric spaces), and let $G$ be such that $D_{\phi G}$ is isometrically embeddable in $Q$ for all $\phi$. Then every minor $H$ of $G$ has this property as well. The ``always $Q$-embeddable'' family of graphs is characterized by a finite list of excluded minors.
\end{thm}
\begin{proof}
    Again, it suffices to consider deletion and contraction of edges. Consider a weighting $\psi$ on $H$. Extend it to $\phi$ on $G$ by setting $\phi(e)=0$ for edges to be contracted, $\phi(e)=N$ sufficiently large for edges to be deleted, and $\phi(e)=\psi(e)$ for edges to be retained. Then $D_{\phi G}$ is simply $D_{\psi G}$ with some duplicate points (as in the proof of Lemma \ref{edgecontracting}. Duplicate points necessarily have the same image in $Q$, so an embedding of $D_{\phi G}$ gives an embedding of $D_{\psi H}$. 
    For the final sentence, we again use the Robertson-Seymour Theorem \cite{ROBERTSON2004325}.
\end{proof}

\section{Proving the main theorem}

\begin{defn}
	Given $S\subseteq \{1,...,n\}\equiv [n]$, we define the \textbf{cut metric} $\delta(S)$ to be the matrix $[a_{ij}]$ with
	\[ a_{ij}:=\begin{cases}
		0 & \mathrm{if} (i\in S \wedge j\in S) \vee (i\notin S \wedge j\notin S)\\ 
		1 & \mathrm{if} (i\notin S \wedge j\in S) \vee (i\in S \wedge j\notin S)
	\end{cases}.\]
\end{defn}	

\begin{defn}
	Given a semimetric $M\equiv ([n], d)$ and sets $A,B\subseteq [n]$, let $\alpha_{M}(A,B)$ equal

 \begin{align*}
		\frac{1}{2}\operatornamewithlimits{\min}_{a,a'\in A, b,b'\in B} \max( d(a,b)+d(a',b')-d(a',a)-d(b',b)&,\\ 
  d(a,b')+d(a',b)-d(a,a')-d(b,b')&,
  \\ 0&).
\end{align*}
	What we'll actually use is $\alpha_M(A):=\alpha_M(A, A^c)$, where $^c$ denotes the complement in $[n]$. Call this quantity the \textbf{isolation index} of $A$. If  $\alpha_M(A)>0$, call $\delta(A)$ an \textbf{$M$-split}. Denote the set of $M$-splits by $\Sigma_M$, and if $\Sigma_M=\emptyset$, we say that $M$ is \textbf{split-prime}.
\end{defn}
\begin{defn}
    Consider a graph $G$ with $V(G)=[n]$ and $S$ a subset of those vertices, consider an edge $e\equiv \overline{ab}$. If $a \in S$ and $b\in S^c$, or if $b\in S$ and $a\in S^c$, we call $e$ a \emph{bridge.}
\end{defn}

When it comes to graphs, we are accustomed to considering partitions of the vertices. Given a weighted graph $\mathcal{G}$, we might say that the splits of $D$ are those partitions ``not worth crossing unless you have to.'' Or, rephrased in rigorous language:

\begin{lmm}\label{collected1} Consider a weighted graph $\phi G$.
	Let $S,S^c$ be a partition of $V(G)$, and let $p_{uv}\subseteq G$ be a path in $G$ from $u$ to $v$. Say it uses $r>1$ bridges. If $\delta(S)$ is a $D_\mathcal{G}$-split, then there exists a path $p'_{uv}$ using $r-2$ bridges with $|p'_{uv}|+2\alpha_{D\mathcal{G}}(S)\leq |p_{uv}|$.  
\end{lmm}
\begin{proof}
	First, we prove it for $r=2$. Assume without loss of generality that $u,v$ lie in $S$. Now, for some $b\in V(p_{uv})\cap S^c$, we have $|p_{uv}|\geq d(u,b)+d(v,b)$. By the definition of $\alpha_D(S)$, $d(u,b)+d(v,b)\geq 2\alpha_D(S)+d(u,v)$, so any $p_{uv}$ using 2 bridges is at least $2\alpha_D(S)$ longer than the shortest path with no bridges. This completes the proof for $r=2$.
	
	Now, take arbitrary $r>1$. Isolate some sequence $x,y,...,y',x'$ that uses exactly 2 bridges: one from $x$ to $y$, another from $y'$ to $x'$. This is an induced subgraph $p_{xx'}$ of $p_{uv}$. Let $p'_{xx'}$ be the shortest path from $x$ to $x'$. By the previous case, $p'_{xx'}$ uses no bridges and has $|p'_{xx'}|+2\alpha_D(S)\leq |p_{xx'}|$. Let $p'_{uv}$ be obtained from $p_{uv}$ by replacing $p_{xx'}$ with $p'_{xx'}$. Thus, $p'_{uv}$ uses $r-2$ bridges and $|p'_{uv}|+2\alpha_D(S)\leq |p_{uv}|$. 
\end{proof}

\begin{lmm}[Part of Theorem 2 from \cite{MR1153934}]\label{Decompose}
	If $M$ is a metric, then
	\[M=M_0+\sum_{\delta(S)\in\Sigma_M}\alpha_M(S)\delta(S),\]
	where $M_0$ is split-prime. 
\end{lmm}

\begin{defn}
	If $M_0=\mathbf{0}$, we say that $M$ is \textbf{totally decomposable}. 
\end{defn}

\begin{defn}
	Given a metric $M$, we say that $M'$ is a \textbf{(distance) minor} of $M$ if $M'$ is a principal submatrix of some $M''$, where $M''$ is of the form
	\[\lambda_0 M_0 +\sum_{\delta(S)\in \Sigma_M}\lambda_S\delta(S)\]
	for some nonnegative $\lambda_0$, $\{\lambda_S\}$, and $M_0$ is the $M_0$ from Lemma \ref{Decompose}. 
\end{defn}
\begin{defn}
	Consider a graph $G$ and a set $S\subseteq V(G)$. The \emph{cut weighting} for $S$ is the function $\Delta(S)$ defined as 
	\[\Delta(S)(e):=\begin{cases}
		1 & \mathrm{if } \;$e$\; \mathrm{\;is\; a\; bridge}\\ 
		0 & \mathrm{otherwise}
	\end{cases}.\]
\end{defn}
The cut-weighting is our original tool, designed to ``play nicely'' with the cut metric. Now we confirm that it does. 

\begin{lmm}\label{cutweighting}
	Let $G$ be a graph, $\phi$ a weighting on $G$, and $\delta(S)$ a $D_{\phi G}$-split. Then $D_{(\phi+x\Delta(S))G}=D_{\phi G}+x\delta(S)$ for $x\in [-\alpha_{D_{\mathcal{G}}}(S), \infty)$.
\end{lmm}
\begin{proof}
As $P_{uv}$ denotes the set of paths from $u$ to $v$, let $P_{uv}^n$ denote the set of paths from $u$ to $v$ that use exactly $n$ bridges. (Obviously, this set will be empty for even $n$ if $u$ and $v$ lie on opposite sides of the partition, and it will be empty for odd $n$ if they lie on the same side.) 

From \ref{collected1}, we have that 
\begin{equation}\label{needed}
	\min_{p\in P_{uv}^k}|p|_{\phi G}\leq \min_{p\in P_{uv}^{(k+2)}}|p|_{\phi G}-2 \alpha_{D_\mathcal{G}}(S)
\end{equation}
and it is clear that for all $p\in P^k_{uv}$,
\begin{equation}\label{1stneeded}
	 |p|_{(\phi+x\Delta(S))G}=|p|_{\phi G}+kx.
\end{equation}
Now, using \ref{needed} and \ref{1stneeded} we have 
\begin{align*}
    \min_{p\in P_{uv}^k} |p|_{\phi G} & \leq \min_{p\in P_{uv}^{k+2}} |p|_{\phi G} -2 \alpha_{D_\mathcal{G}}(S) \\
    (k+2)x+ \min_{p\in P_{uv}^k} |p|_{\phi G} & \leq (k+2)x+ \min_{p\in P_{uv}^{k+2}} |p|_{\phi G} -2 \alpha_{D_\mathcal{G}}(S)\\
   2x+ \min_{p\in P_{uv}^k} |p|_{(\phi+x\Delta(S)) G} & \leq \min_{p\in P_{uv}^{k+2}}|p|_{(\phi+x\Delta(S)) G} -2 \alpha_{D_\mathcal{G}}(S) \\ 
     \min_{p\in P_{uv}^k} |p|_{(\phi+x\Delta(S)) G} & \leq \min_{p\in P_{uv}^{k+2}}|p|_{(\phi+x\Delta(S)) G} -2(x+\alpha_{D_\mathcal{G}}(S)).
\end{align*}
Therefore, since $x\in [-\alpha_{D_{\mathcal{G}}}(S), \infty)$, we have that the shortest path (or some path tied for shortest) between any two vertices in the $\phi+x\Delta(S)$ weighting of $G$ uses the fewest bridges possible. That is, 0 or 1 bridges. Lemma \ref{collected1} implies that this was the case in the $\phi$ weighting as well. Therefore, the old shortest paths, lying in $P_{uv}^0$ if $u$ and $v$ are on the same side of the partition or in $P_{uv}^1$ if they are not, are the new shortest paths. If $u,v$ are on the same side of the partition, the lengths of the edges in the path are completely unchanged, so  $d_{(\phi+x\Delta(S))G}(u,v)=d_{\phi G}(u,v)$. If they are on the opposite sides of the partition, the path uses exactly one bridge, so the length of the path has been changed by $x$; i.e. $d_{(\phi+x\Delta(S))G}(u,v)=d_{\phi G}(u,v)+x$. By the previous two sentences, we have  $d_{(\phi+x\Delta(S))G}(u,v)=d_{\phi G}(u,v)+x\delta(S)(u,v)$, completing the proof that $D_{(\phi+x\Delta(S))G}=D_{\phi G}+x\delta(S)$.
\end{proof}

\begin{lmm}[A well-known graph fact]\label{WellKnownGraphFact}
If $H$ is a graph with maximum degree $\leq 3$, then for all graphs $G$, $G$ contains an $H$ minor if and only if it contains an $H$ subdivision.
\end{lmm}

\begin{thm}\label{k23theorem}
	Let $G$ be a graph. The following are equivalent:
	\begin{enumerate}[label=(\roman*)]
		\item $G$ has no $K_{2,3}$ (graph) minor.
		\item $\nexists \phi: E(G) \rightarrow [0,\infty)$, $c\in \mathbb{R}$  s.t. $D_{cK_{2,3}}$ is a (distance) minor of $D_{\phi G}$.
		\item $\forall  \phi: E(G) \rightarrow [0,\infty)$, $D_{\phi G}$ is totally decomposable.
		\item $\forall \phi: E(G) \rightarrow [0,\infty)$, $D_{\phi G}$ is $\ell_1$-embeddable.
		\item $\forall \phi: E(G) \rightarrow [0,\infty)$, $i_+(D_{\phi G})=1$. 
	\end{enumerate}
\end{thm}

\begin{proof}\label{k23proof}
We prove $(i)\rightarrow(ii)\rightarrow(iii)\rightarrow(iv)\rightarrow(v)\rightarrow(i)$. 

$(i)\rightarrow (ii)$ We prove by contrapositive. Let $\phi: E(G) \rightarrow [0,\infty)$ be such that $D_{K_{2,3}}$ is a distance minor of $D_{\phi G}$. Define \[\psi:=\phi-\sum_{\delta(S)\in \Sigma_{D_{\phi G}}} \alpha_{D_{\phi G}}(S) \Delta(S)).\] 

(We have $n-1$ instead of $n$ to avoid symmetry problems.) By Lemma \ref{cutweighting}, Lemma \ref{Decompose}, and the definition of a $D_{\psi G}$-split, we see that $D_{\psi G}$ is split-prime. By definition 11.1.12 and fact 11.1.13 of \cite{MR2841334}, there exists a positive $c$ such that $D_{cK_{2,3}}$ is a principal submatrix of $D_{\psi G}$. Thus, there exist $x_1,x_2,x_3,y_1,y_2\in V(G)$ such that the shortest path from $x_i$ to $y_j$ has length $c$ (weighted by $\psi$), and these paths are vertex-disjoint (up to identification of vertices joined by an edge of weight 0), since otherwise the shortest path from $x_i$ to $x_j$, $i\neq j$, would have length $<2c$. (Also there must not be other, shorter paths between $x_i$ and $x_j$ or between $y_i$ and $y_j$ for $i\neq j$, but that's not so important for the argument.)
Therefore we have a $K_{2,3}$ subdivision in $G$, which is a $K_{2,3}$ graph minor of $G$. 

$(ii)\rightarrow (iii)$ Apply fact 11.1.13 from \cite{MR2841334}.\footnote{Contrary to what \cite{MR2841334} implies, this fact isn't stated as such in \cite{MR1153934}, the original source. Instead, it's an elementary consequence of Theorem 6 and Lemma 1.}

$(iii) \rightarrow (iv)$ For any totally decomposable distance $M$ on $[n]$, assign to each $S\subseteq [n-1]$ a coordinate. For every member of $S$, set that coordinate to equal $\alpha_M(S)$. For everything else, set that coordinate to 0. Inspection reveals that this is an isometric embedding. 

$(iv)\rightarrow (v)$ From Theorem 6.3.1 of \cite{MR2841334}, we have that all nonzero $\ell_1$-embeddable matrices $M$ have $i_+(M)=1$. 

$(v)\rightarrow (i)$ We prove by contrapositive. Let $G$ have a $K_{2,3}$ minor. Then, by Lemma \ref{WellKnownGraphFact}, some subdivision of $K_{2,3}$ is a subgraph of $G$. Let $\phi: E(G) \rightarrow [0,\infty)$ be such that the paths from one original vertex of $K_{2,3}$ to another have length 1. Let all other edges have weight $x\geq2$. Then the matrix \begin{center}
    $D_{K_{2,3}}\equiv$
$\begin{pmatrix}
    0 & 2 & 2 & 1 & 1 \\
    2 & 0 & 2 & 1 & 1 \\
    2 & 2 & 0 & 1 & 1 \\
    1 & 1 & 1 & 0 & 2 \\
    1 & 1 & 1 & 2 & 0 \\
    \end{pmatrix}$
\end{center}
is a principal submatrix of $D_{\phi G}$.  Both matrices are Hermitian, and we may calculate that $i_+(D_{K_{2,3}})= 2$, so by Cauchy's Interlacing Theorem, $i_+(D_{\phi G})\geq 2$. This completes the contrapositive and the proof. 
 \end{proof}

\begin{rem}
    Long after writing this section, we discovered that \cite{CutsTrees} had already proven the equivalence of $(i)$ and $(iii)$. Their proof is quite different; it uses the following facts.
    \begin{itemize}
        \item $K_{2,3}$-minor-free graphs are precisely those whose 2-connected components are either outerplanar or isomorphic to $K_4$. 
        \item The MaxFlow and MinCut values are equal for outerplanar graphs. 
        \item The MaxFlow/MinCut gap is equivalent to the distortion incurred when embedding in $\ell_1$. 
        \item Two weighted graphs isometrically embeddable in $\ell_1$, joined at a single point, form a weighted graph isometrically embeddable in $\ell_1$. 
        \item An unweighted $K_{2,3}$ is not isometrically embeddable in $\ell_1$.
    \end{itemize} 
    Our proof foregrounds the algebraic aspect more, and it might give better clues to the specialness of $K_{2,3}$.
\end{rem}

\section{The K-two-threes conjecture}

Experimentation suggests the following. If true, it is tight, by Lemma \ref{ktwothreesforbidden}.

\begin{conj}[The weak K-two-threes conjecture]\label{WeakConjecture}
    If $M$ is an $n$-by-$n$ metric matrix,\footnote{I.e., a hollow, symmetric, nonnegative matrix with $m_{jk}\leq m_{ij}+m_{jk}$ for all $1\leq i,j,k\leq n$. Alternatively, a matrix $M$ such that there exists a $G,phi$ such that $M=M_{\phi G}$.} then $M$ has $\leq \lfloor \frac{n+1}{3}\rfloor$ positive eigenvalues.
\end{conj}

By Corollary \ref{InertialRestrictionMinorClosed}, the graphs with $\leq n$ positive eigenvalues can be described by finitely many forbidden minors, which must themselves by graphs $G$ such that, for some $\phi$, $M\phi G$ has $>n$ positive eigenvalues. 

\begin{lmm}\label{ktwothreesforbidden}
    The graph $K_{2,3,...,3}$ (with $k$ threes) has unweighted distance inertia $(k+1,0,2k+1)$.
\end{lmm}
\begin{proof}
One can prove this by repeatedly applying the construction in Lemma 2.2 of \cite{MR364288} to $k$ copies of the matrix $2(J_3-I_3)$ and one copy of the matrix $2(J_2-I_2)$, or one can simply invoke Theorem 3.3 of \cite{ZHANG2014108}
\end{proof}

This motivates the prettier Conjecture \ref{StrongConjecture}. Its $n=1$ case is the $(i) \leftrightarrow (v)$ of Theorem \ref{k23theorem}. 

\begin{conj}[The strong K-two-threes conjecture]\label{StrongConjecture}
$D_{\phi G}$ has $\leq n$ positive eigenvalues for all $\phi$ if and only if $G$ has no $K_{2,3,...,3}$ minor (with $n$ threes). 
\end{conj}

The strong version implies the weak version, as it ought to. 

Furthermore, it is known that a matrix embeddable in some $\ell_p$ for $p\leq 2$ has at most one positive eigenvalue. Computer experimentation with generating random $\ell_p$-embeddable matrices suggest that for any $p>2$, there exist $\ell_p$-embeddable matrices with more than one positive eigenvalue, but that as $p$ approaches $2$, the necessary matrix size for this to occur tends to infinity.

\printbibliography

\end{document}